\documentclass[10pt]{article}

\usepackage{fullpage}
\usepackage[dvipsnames]{xcolor}  % Coloured text etc.
\usepackage{tipa}
\usepackage{dsfont}
\usepackage{tikz-cd}
\usepackage{tikz, tikz-3dplot, pgfplots}
\usetikzlibrary{decorations.markings, patterns, decorations.pathmorphing}

\tikzset{->-/.style={decoration={
markings,
mark=at position #1 with {\arrow{>}}},postaction={decorate}}}

\usepackage{graphicx}

\usepackage{amsxtra}
\usepackage{amsmath}
\usepackage{amssymb}
\usepackage{amsfonts}
\usepackage{mathrsfs}
\usepackage{amsthm}
\usepackage[all]{xy}
\usepackage{enumitem}
\usepackage{hyperref}
\usepackage{subfigure}
\usepackage{mathabx}
\usepackage{euscript}
\usepackage{xargs}                      % Use more than one optional parameter in a new commands

\usepackage[colorinlistoftodos,prependcaption,textsize=tiny]{todonotes}
\newcommandx{\note}[2][1=]{\todo[linecolor=Plum,backgroundcolor=Plum!25,bordercolor=Plum,#1]{#2}}

\newtheoremstyle{example}{\topsep}{\topsep}%
     {}%         Body font
     {}%         Indent amount (empty = no indent, \parindent = para indent)
     {\bfseries}% Thm head font
     {.}%        Punctuation after thm head
     {2pt}%     Space after thm head (\newline = linebreak)
     {\thmname{#1}\thmnumber{ #2}\thmnote{ #3}}%         Thm head spec

\theoremstyle{example}
\newtheorem{exa}[equation]{Example}

\newtheorem{rem}[equation]{Remark}

\newtheorem{defi}[equation]{Definition}

\newtheorem{thm}[equation]{Theorem}
\newtheorem{introthm}[section]{Theorem}
\newtheorem{cor}[equation]{Corollary}
\newtheorem{lem}[equation]{Lemma}
\newtheorem{prop}[equation]{Proposition}

\hypersetup{%
    colorlinks,%
    linkcolor={red!50!black},%
    citecolor={blue!50!black},%
    urlcolor={blue!80!black}%
}

\setlist[enumerate,1]{label=(\arabic{*})}
\setlist[enumerate,2]{label=(\roman{*})}
\setlist[enumerate,3]{label=(\alph{*})}

\def\B{ {\EuScript B}}

\def\C{{\EuScript C}}

\def\F{{ \EuScript F}}

\def\Fun{\operatorname{Fun}}

\def\Fin{\operatorname{Fin}}
\def\fin{\operatorname{fin}}

\def\Hom{\operatorname{Hom}}

\def\Hyp{\operatorname{Hyp}}
 
 \def\im{\on{im}}

\def\O{{\mathcal O}}

\def\on{\operatorname}

\def\op{{\operatorname{op}}}

 \def\P{{\EuScript P}}

\def\phi{{\varphi}}

 \def\Set{ {\operatorname{Set}}}

\def\Sh{\on{Sh}}

\def\U{ {\EuScript U}}

 \def\={{\,\, \simeq\,\, }}
 \def\-{{\setminus}}
\def\<<{\langle {}\hskip -.1cm {}\langle}
\def\>>{\rangle \hskip -.1cm \rangle}

\title{Hypersheaves and bases}
\author{Tobias Dyckerhoff, Till Heine, Simon Schneider}
\begin{document}

\maketitle

\begin{abstract}
    Let $X$ be a topological space equipped with a basis. We prove that, for
    every $\infty$-category $\C$ with limits, the restriction functor from
    $\C$-valued hypersheaves on $X$ to basic hypersheaves is an equivalence of
    $\infty$-categories. 
\end{abstract}

\tableofcontents

\addcontentsline{toc}{section}{Introduction}

\numberwithin{equation}{section}

\newpage
\section*{Introduction}

The Zariski topology on an affine scheme $\on{Spec}(R)$ has a basis given by
the collection of affine opens $U_f = \on{Spec}(R_f)$ where $f \in R$. An
elegant way to construct the structure sheaf of $\on{Spec}(R)$ is to first
specify its values $U_f \mapsto R_f$ on this basis and then Kan extend to all
opens. 
This hinges upon the general fact that, given any category $\C$ with limits and
a topological space $X$ with basis $\B$, closed under finite intersections, the
category of $\C$-valued sheaves on $X$ is, via restriction, equivalent to the
category of $\C$-valued sheaves on $\B$. This statement remains true literally
when passing to the homotopy coherent context (i.e. replacing $\C$ by an
$\infty$-category), and can be utilized to construct the structure sheaf of an
affine spectral scheme in complete analogy to the above construction for affine
schemes (cf. \cite[1.1.4]{lurie:sag}). 

The requirement that the basis $\B$ be intersection closed is rather restrictive
and not satisfied in many situations of interest. For example, it fails for the
basis of disks in a topological manifold. In ordinary category theory,
this can be fixed by a minor modification of the concept of a sheaf on the
basis $\B$: For every basic covering $\{B_i\}$ of a basic open $B$, along with
basic coverings $\{B_{ijk}\}$ of each intersection $B_i \cap B_j$, we require
that sections over $B$ correspond uniquely to families of sections over
$\{B_i\}$ that agree, for each intersection $B_i \cap B_j$, locally on
$\{B_{ijk}\}$. In other words, the diagram 
\begin{equation}\label{eq:set_sheaf_basis_condition}
\F(B)  \to \prod \F(B_i) \rightrightarrows \prod \F(B_{ijk} )
\end{equation}
is an equalizer cone. Using this notion of a sheaf on $\B$, for an ordinary category
$\C$, the desired equivalence between $\C$-valued sheaves on $X$ and sheaves on
$\B$ remains intact. 

However, when moving to the homotopy coherent context, things become more
subtle. Namely, the truncated descent condition
\eqref{eq:set_sheaf_basis_condition} should be replaced by a simplicial
analogue also encoding basic coverings of $n$-fold intersections, $n > 2$. This
leads to the classical concept of a {\em hypercovering} (cf. \cite{verdier:SGA4V})
which may be formulated as follows: A basic hypercovering consists of a simplicial
set $K$ along with a functor $\U:\Delta/K \to \B^{\op}$ such that, for every
$\tau:\partial \Delta^n \to K$, we have 
\[
    \bigcup_{\substack{\sigma \in K_n\\\sigma|\partial\Delta^n=\tau}}
    \U(\sigma)=\bigcap_{0 \le i \le n} \U(\partial_i
\tau).
\]
The corresponding {\em hyperdescent} condition then amounts to 
\[
    \F(B) \simeq \lim_{\Delta/K} \F \circ \U
\]
and we refer to a functor $\F: \B^{\op} \to \C$ satisfying it as a $\C$-valued
{\em hypersheaf} on $\B$. We obtain an analogous hyperdescent condition for
presheaves on $X$, simply replacing $\B$ by the poset of {\em all} opens in the above
definitions. 
In this context, the main result proven in this note is the following:

\begin{introthm}
    \label{thm:main}
    Let $\C$ be an $\infty$-category with limits, $X$ a topological space, and
    $\B$ a basis for the topology of $X$. Then the $\infty$-category of
    $\C$-valued hypersheaves on $X$ is equivalent, via restriction, to the
    $\infty$-category of $\C$-valued hypersheaves on $\B$. 
\end{introthm}

Our proof is a typical ``lossless data transfer by means of Kan extension''
argument and is therefore based on rather explicit cofinality arguments.
Two essential techniques are {\em symmetrization}, allowing us to replace the
simplex category $\Delta$ by the category of standard finite sets $\Fin$, and
{\em categorification}, allowing us to replace sets by categories. 

Theorem \ref{thm:main} is used in an ongoing project on constructible sheaves
-- we were not able to find a reference on the desired level of generality.
However, as Markus Zetto pointed out to us after the completion of our proof,
in \cite{aoki:ttg}, the special case where $\C$ is the $\infty$-category of
spaces is proven using topos-theoretic methods and the possibility of a purely
combinatorial proof is also mentioned and attributed to Jacob Lurie. 

For sheaves (instead of hypersheaves), there does not seem to be a general
analogue of the Theorem. There is a version of the sheaf condition that does
not require intersection stability (see Definition \ref{def:sheaf_basis}), but
even with this notion, sheaves will not generally coincide with sheaves on a
basis (e.g. the Hilbert cube, cf. \cite[3.12]{bgh:exodromy}). Nevertheless, in
various situations of interest, sheaves and hypersheaves agree (cf.
\cite[7.2.1.12]{lurie:htt}).

\subsection*{Acknowledgements}
We would like to thank Markus Zetto for helpful discussions.

\subsection*{Funding}
The authors acknowledge support from the Deutsche
Forschungsgemeinschaft (DFG, German Research Foundation) under Germany's
Excellence Strategy - EXC 2121 ``Quantum Universe'' - 390833306 and the
Collaborative Research Center - SFB 1624 ``Higher structures, moduli spaces and
integrability'' - 506632645.

\setcounter{section}{0}

\section{Conventions}%
\label{sec:conventions}

Our standard reference for homotopy coherent mathematics is \cite{lurie:htt}. In
particular, by an {\em $\infty$-category} we mean a simplicial set satisfying the weak
inner horn filling condition. Ordinary categories will be implicitly identified
with their nerves.
 
A functor $f: A \to B$ of (small) $\infty$-categories is called {\em cofinal} if, for every object $b \in B$, the geometric realization of the slice category
\[
    b/f = A \times_{B} B_{b/}
\]
under $b$ is contractible. In particular, this means that an object $x \in B$
is a final object if and only if the inclusion $\{x\} \to B$ is cofinal.
Cofinal functors preserve colimits. 

Dually, $f$ is called {\em coinitial} if, for every object $b \in B$, the
geometric realization of the slice category
\[
    f/b = A \times_{B} B_{/b}
\]
over $b$ is contractible. Note that an object $x \in B$ is an initial object if and only
if the inclusion $\{x\} \to B$ is coinitial. Coinitial functors preserve limits. 

Throughout, we fix a topological space $X$, and denote by $\B$ a basis of its topology.

\section{Sheaves on intersection stable bases}
\label{sec:sheaves_basis_intersection}

In this section, we review that, if $\B$ is intersection-stable, the notion of
sheaves on a basis does agree with that of sheaves on the space. The arguments
also serve as a rough blueprint for the more technical proof of the theorem for
hypersheaves, and are largely the same as the ones that appear in
\cite[Proposition 1.1.4.4]{lurie:sag}.
Let $X$ be a (small) topological space and denote by $\O(X)$ the poset of its
open subsets. Let $\C$ be an $\infty$-category with (small) limits.

\begin{defi}
    \label{def:sheaf_basis}
    Let $\B \subset \O(X)$ be a basis of the topology. A $\C$-valued presheaf
    \[
    \F:\B^{\op} \to \C
    \]
    is a sheaf on $\B$ if, for every $B \in \B$ and every open cover
    $\{B_i\}_{i \in I}$ of $B$ by elements of $\B$, restriction induces an
    equivalence
    \[
	\F(B) \simeq \lim_{U \in (\B/\{B_i\})^{\op}} \F(U),
    \]
    where $\B / \{B_i\} \subset \B$ is the poset spanned by sets $U \in \B$
    such that $U \subset B_i$ for some $i \in I$.
    We denote by
    \[
    \Sh(\B,\C) \subset \Fun(\B^{\op},\C)
    \]
    the $\infty$-category of $\C$-valued sheaves on $\B$. For $\B=\O(X)$ we also
    write $\Sh(X,\C)$ and refer to these as sheaves on $X$.
\end{defi}

\begin{lem}\label{lem:cof}Let $U \subset X$ be an open set and $\{V_i\}_{i \in I}$ an open cover of $U$. For $S=\{i_1, \dots, i_n\} \subset I$, we denote $V_S=V_{i_1} \cap V_{i_2} \dots \cap V_{i_n}$. The inclusion of posets
\[
    \P_{\fin}(\{V_i\})=\{V_S | S \subset I \text{ finite} \} \subset \O(X)/\{V_i\}
\]
is cofinal.
\end{lem}
\begin{proof}
    For any $V \in \O(X)/\{V_i\}$, the category $V/\P_{\fin}(\{V_i\})$ is filtered, hence weakly contractible.
\end{proof}
\begin{lem}
Let $X$ be a topological space and let $\B$ be a basis of the topology that is stable under intersections. Then the restriction map
\[
\Sh(X,\C)\to \Sh(\B,\C)
\]
is well-defined and an equivalence of $\infty$-categories.
\end{lem}
\begin{proof}
We first prove well-definedness. Let $\F \in \Sh(X,\C)$ be a sheaf. Given a subset $B \in \B$ and a cover $\{B_i\}_{i \in I}$ of $B$ by open sets belonging to $\B$, we must show that restriction induces an equivalence
\[
    \F(B) \simeq \lim_{V \in (\B/\{B_i\})^{\op}}\F(V).
\]
It follows from the sheaf condition on $X$ that restriction induces an equivalence
\[
    \F(B) \simeq \lim_{V \in (\O(X)/\{B_i\})^{\op}}\F(V).
\]
By Lemma \ref{lem:cof} and $\cap$-stability of $\B$, these conditions are equivalent since $\P_{\fin}(\{B_i\})$ is a cofinal subcategory of both $\B/\{B_i\}$ and $\O(X)/\{B_i\}$.

We now show that the restriction functor is an equivalence by showing that sheaves on $X$ are precisely the right Kan extensions of sheaves on $\B$.
Suppose first that $\F:\O(X)^{\op} \to \C$ is a sheaf on $X$. By the pointwise formula, we must show that, for $U \in \O(X)$, the restriction
\[
    \F(U) \to \lim_{V \in (\B/U)^{\op}} \F(V)
\]
is an equivalence. By regarding $\B/U$ as an open cover of $U$, we see that $\F$ satisfies
\[
    \F(U) \simeq \lim_{V \in (\O(X)/(\B/U))^{\op}} \F(V).
\]
This condition is again equivalent to the desired formula by Lemma \ref{lem:cof} and $\cap$-stability of $\B$.

Conversely, suppose that $\F:\O(X)^{\op} \to \C$ is a right Kan extension of its restriction to $\B^{\op}$ and that $\F|{\B^{\op}}$ is a sheaf. Given an open subset $U \subset X$ and an open cover $\{U_i\}_{i \in I}$, we must show that restriction induces an equivalence
\begin{equation}
\label{eq:proofshcon}
\F(U) \simeq \lim_{V \in (\O(X)/\{U_i\})^{\op}} \F(V).
\end{equation}
We have the following diagram of index categories:
\[
\begin{tikzcd}
(\B/\{U_i\})^{\op} \arrow[r, "i"] \arrow[d] & (\B/U)^{\op} \arrow[d]          \\
(\O(X)/\{U_i\})^{\op} \arrow[r, "j"]           & (\O(X)/U)^{\op}
\end{tikzcd} 
\]
By assumption, $\F$ is a right Kan extension along the inclusion $\B^{\op} \to \O(X)^{\op}$. It follows that the restriction of $\F$ is also a right Kan extension along the vertical inclusions in the above diagram. We shall now prove that the restriction of $\F$ is also a right Kan extension along $i$:

For $B \in \B/U$, the basis opens $\{\B/\{U_i \cap B\}\} =\{B' \in \B/\{U_i\} | B' \subset B\}$ form an open cover of $B$, since $\{U_i\}$ is an open cover of $U \supset B$ and $\B$ is a basis. It therefore follows from the sheaf condition on $\B$ that
\[
    \F(B) \simeq \lim_{V \in (\B/\{U_i \cap B\})^{\op}} \F(V),
\]
because $\B/\{U_i \cap B\}$ is downward-closed  by definition. This is precisely the pointwise formula which witnesses that the restriction of $\F$ is a right Kan extension along $i:(\B/\{U_i\})^{\op} \to (\B/U)^{\op}$.

Finally, it follows from the dual of \cite[4.3.2.8]{lurie:htt}  that the restriction of $\F$ is a right Kan extension along $j$ as well. The pointwise formula for this extension now yields \eqref{eq:proofshcon}.
\end{proof}

\section{Symmetrization of simplicial sets}
\label{sec:symmetrization}
We will utilize a {\em symmetrization} construction for simplicial sets,
which we now analyze. Let $\Fin$ denote the category with an object $\langle
n \rangle$ for each $n \geq 0$ and morphisms $\langle n \rangle \to \langle m
\rangle$ given by maps $\{0, \dots,n\} \to \{0, \dots,m\}$. For a
simplicial set $K$, we denote by $S(K): \Fin^{\op} \to \Set$ the left Kan
extension of $K$ along the inclusion $\Delta^{\op} \subset \Fin^{\op}$.
Explicitly, we have
\[
    S(K)_n \cong ( \coprod_{f:\langle n \rangle \to \langle m \rangle} K_m )_{/\sim}
\]
where the relation is generated by the identifications $(\sigma,g \circ f)
\sim (g^{*}(\sigma),f)$, $g$ ranging over the set of weakly increasing maps.

The following notation will be useful. For a map $f: \langle n \rangle \to
\langle m \rangle$ in $\Fin$, we write $f=f_i \circ f_s$ for the unique
factorization with $f_i$ strictly increasing and $f_s$ surjective.
Given $\sigma \in K_n$, we write $(\bar{\sigma},s_{\sigma})$ for the
unique pair given by $\bar{\sigma} \in K_m$ nondegenerate and $s_{\sigma}:[n] \to [m]$
surjective, such that $s_{\sigma}^*(\overline{\sigma})=\sigma$.

\begin{lem}
    \label{lem:minimal}
    Let $K$ be a simplicial set and $n \ge 0$. Then every equivalence class in
    $S(K)_n$ has a unique representative $(\sigma,f)$ with $\sigma$
    nondegenerate and $f$ surjective.
\end{lem}
\begin{proof}
    Let $M_n$ denote the set of pairs $(\sigma,f)$ with $\sigma \in K_m$
    nondegenerate and $f: \langle n \rangle \to \langle m \rangle$ surjective.
    We will show that the canonical map $M_n \to S(K)_n$ is an isomorphism by
    constructing an inverse
    \[
	R:S(K)_n \to M_n.
    \]
    Given an arbitrary representative
    $(\sigma,f)$ of an element of $S(K)_n$, we have
	\[
		(\sigma, f) \sim (f_i^* \sigma, f_s) \sim (\overline{f_i^*\sigma},s_{f_i^* \sigma} \circ f_s).
	\]
    It is clear that the assignment
    $R(\sigma,f)=(\overline{f_i^*\sigma},s_{f_i^* \sigma} \circ f_s)$ defines
    an inverse if it descends to the quotient. Suppose that $g:[m] \to [l]$
    is weakly increasing. We must show that $R$ sends $(\sigma,g \circ f)$ and
    $(g^* \sigma,f)$ to the same element of $M_n$. We calculate
	\[
        R(\sigma,g \circ f) = (\overline{(g \circ f)_i^* \sigma},s_{(g \circ f)_i^* \sigma} \circ (g \circ f)_s). 
    \]
    For convenience, we abbreviate $\tau=\overline{(g \circ f)_i^* \sigma}$ and
    $s=s_{(g \circ f)_i^* \sigma}$. Note that $(g \circ f_i)_s$ is weakly
    increasing, since $g$ and $f_i$ are. We obtain
	\[
		(s \circ (g \circ f_i)_s)^*\tau =(g \circ f_i)_s^*(g \circ f)_i^* \sigma =(g \circ f_i)^* \sigma,
	\]
    where the first equality comes from the definition of $\tau$ and $s$, and
    the second comes from the fact that $(g \circ f)_i = (g \circ f_i)_i$. We
    also have
	\[
		s \circ (g \circ f_i)_s \circ f_s=s \circ (g \circ f)_s,
	\]
    which shows that $(\sigma,g \circ f)$ and $(g^*\sigma,f)$ do indeed have
    the same image under $R$.
\end{proof}

We refer to the unique representative from Lemma \ref{lem:minimal} as a {\em minimal} representative. 

\begin{lem}
    \label{lem:fin_top}
    Let $f:\langle n \rangle \to \langle l \rangle$ be a map of finite sets. Then 
    the simplicial set $C^f$ defined by
    \[
        C^f_m=\{g: \langle m \rangle \to \langle n \rangle\; |\;  f \circ g \text{ weakly increasing} \}
    \]
    with functoriality via precomposition, is weakly contractible.
\end{lem}
\begin{proof}
Choose $i \in \{0, \dots, n\}$ such that $f(i) \geq f(j)$ for all $j \in \{0, \dots,n\}$. For every $g \in C^f_m$, we define the morphism $\tilde{g}: \langle m+1 \rangle \to \langle n \rangle$ by $\tilde{g} \circ \partial_{m+1}=g$ and $g(m+1)=i$. The map $f \circ \tilde{g}$ is then also weakly increasing so that $\tilde{g} \in C^f_{m+1}$. We now describe a contraction of the topological space $|C^f|$ onto the point $i$: For each $g \in C^f_m$ there is a homotopy
\[
H_g: [0,1] \times |\Delta^m_g| \to |\Delta^{m+1}_{\tilde{g}}| \to |C^f|
\]
from the natural inclusion in the image of $|\Delta^{m+1}_{\tilde{g}}|$ to $i$ via straight lines. These homotopies descend to the quotient to yield a contraction of $|C^f|$.
\end{proof}

\begin{prop}
    \label{prop:sym_unit_coinit}
    Let $K$ be a simplicial set and $S(K)$ its symmetrization. Then the functor
	\[
		\Delta/K \to \Fin/S(K)
	\]
	is coinitial.
\end{prop}
\begin{proof}
    Applying Quillen's Theorem A, we need to show that, for $[\sigma,f] \in
    S(K)_n$, which we assume to be given by a minimal representative, the category
    $(\Delta/K)/[\sigma,f]$ is weakly contractible. This category can be
    identified with the category of simplices of the simplicial set
    $H^{[\sigma,f]}$, whose $m$-simplices are given by pairs
    \[
	(g:\langle m \rangle \to \langle n \rangle,\tau \in K_m)
    \]
    such that $[\tau,id]=[\sigma,f
    \circ g]$.
    Since any equivalence class has at most one representative of the form
    $[\tau,id]$, this represents no additional data, and we have an inclusion
    $H^{[\sigma,f]} \subset \Delta/\Fin(-,\langle n \rangle)$.
	
    We show that $|H^{[\sigma,f]}|$ is contractible for each $[\sigma, f] \in
    S(K)_n$, by induction on $n$. The base case is trivial, so assume the claim is
    true for all $m < n$. Let $J = (\Delta^1 \times \P([n])\backslash\{\emptyset\})
    \backslash \{(1,[n])\}$. We will describe a covering
    \[
    F:J \to \O(|H^{[\sigma,f]}|)
    \]
    of $|H^{[\sigma,f]}|$ by contractible subspaces, that satisfies the
    assumptions of Lurie's van-Kampen Theorem \cite[Theorem A.3.1]{lurie:ha}. As $J$ is cofiltered, hence
    weakly contractible, this will imply the claim.

    For $A \subset \{0, \dots, n\}$, let $H_{(1,A)} \subset H^{[\sigma,f]}$
    denote the simplicial subset with
    \[
    (H_{(1,A)})_m=\{g \in H^{[\sigma,f]}_m | \im g \subset A\}.
    \]
    Let further $H_{(0,A)} \subset H_{(1,A)}$ denote the simplicial subset with
    \[
    (H_{(0,A)})_m= \{g: \langle m \rangle \to \langle n \rangle | \im g \subset A \text{ and } f \circ g \text{ is weakly increasing} \}.
    \]
    These simplicial subsets provide a covering
    \[
        H^{[\sigma,f]} = \bigcup_{A \subsetneq [n]} H_{(1,A)} \cup H_{(0,[n])}.
    \]
    Indeed, if a given $g \in H^{[\sigma,f]}_m$ is not contained in any of the
    $(H_{(1,A)})_m$ for $S \subsetneq [n]$, it is necessarily surjective.
    Therefore $[\sigma,f \circ g]$ is its own minimal representative and this
    class is equal to some $[\tau,id]$ if and only if $f \circ g$ is weakly
    increasing, i.e. $g \in (H_{0,[n]})_m$.

    For each nonempty $A \subset [n]$, let $\iota_A:[|A|-1] \to [n]$ denote the
    increasing map with image $A$. For a proper subset $A$, the assignment $h
    \mapsto \iota_A \circ h$ describes an isomorphism $K^{[\sigma,f \circ
    \iota_A]} \cong K_{(1,A)}$ and the former simplicial set is weakly
    contractible by inductive hypothesis. Furthermore, for any subset $A$, the
    above assignment induces an isomorphism $C^{f \circ \iota_A} \cong
    H_{(0,A)}$, and we have shown in Lemma \ref{lem:fin_top} that the former is
    weakly contractible.

    Recall that any subcomplex $C$ of a CW-complex $X$ has an open neighborhood
    $C_{\epsilon} \subset X$ of which it is a deformation retract, and that
    such an open neighborhood construction can be chosen to commute with
    intersections.
    The diagram
    \[
    F(b,A)=|H_{b,A}|_{\epsilon}
    \]
    then describes a covering of $|H^{[\sigma,f]}|$ by contractible opens. It
    is clear from the $\cap$-stability of the $(-)_{\epsilon}$-construction
    that $F^{-1}(x)$ is cofiltered, hence weakly contractible, for any point $x
    \in |H^{[\sigma,f]}|$.  We may therefore apply Lurie's van-Kampen theorem
    to conclude that $|H^{[\sigma,f]}|$ is contractible, as desired.
\end{proof}

\begin{cor}
The inclusion $\Delta \to \Fin$ is coinitial.
\end{cor}
\begin{proof}
This is the case $K=\Delta^0$.
\end{proof}

\begin{cor}
\label{cor:forget_fin}
For a functor $H:\Fin^{\op} \to \Set$, the functor $\Delta/H \to \Fin/H$ is coinitial.
\end{cor}
\begin{proof}
There is a pullback square
\[\begin{tikzcd}
	{\Delta/H} & {\Fin/H} \\
	\Delta & \Fin
	\arrow[from=1-1, to=1-2]
	\arrow[from=1-1, to=2-1]
	\arrow[from=1-2, to=2-2]
	\arrow[from=2-1, to=2-2]
\end{tikzcd}\]
where the bottom arrow is coinitial by the above corollary and the vertical arrows are Cartesian.
 The conclusion follows from \cite[Proposition 4.1.2.14]{lurie:htt}.
\end{proof}

\begin{rem}
The coinitiality of $\Delta \to \Fin$ on its own is not sufficient to formally conclude the statement of Proposition \ref{prop:sym_unit_coinit}. For example, the inclusion of posets
\[
I=\Delta^{\{0,1\}} \coprod_{\Delta^{\{1\}}} \Lambda_0^2 \to J=\Delta^2 \coprod_{\Delta^{\{1,2\}}} \Delta^2
\]
is coinitial, but, letting $v$ denote the final vertex of $J$, $I/\Hom_I(-,v) \to J/\Hom_J(-,v)$ is not coinitial: Indeed, one of the relevant comma categories is empty by inspection.
\end{rem}

\section{Hyperdescent}
\label{sec:Hyperdescent}

\begin{defi}
  \label{def:hypercovering}
  A {\em hypercovering $(K, \U)$ of $X$}
  consists of 
  \begin{enumerate}[label=\arabic*.]
      \item a simplicial set $K$, 
      \item a functor $\U: \Delta/K \to \O(X)^{\op}$, where $\Delta/K$ denotes
          the category of simplices of $K$,
    \end{enumerate}
  subject to the following condition
    \begin{itemize}
      \item for every $n \geq 0$ and every
        \[
            \tau: \partial\Delta^n \to K,
        \]
        we have
        \begin{equation}
          \label{eq:hypercovering}
          \bigcup_{\substack{\sigma \in K_n\\\sigma|\partial \Delta^n = \tau}} \U_{\sigma}= \bigcap_{i = 0}^{n} \U_{\partial_i \tau}.
        \end{equation}
    \end{itemize}
\end{defi}

\begin{rem}
	\label{rem:addendum_hypercover}
	\begin{enumerate}
		\item For $n = 0$, the condition \eqref{eq:hypercovering} reads
		\[ 
			\bigcup_{i \in K_0} \U_i = X 
		\]

		\item We will typically consider, for a fixed topological space
			$X$, hypercoverings of various open subsets $U \subset
			X$. In this case, we consider $U$ as a topological
			space with the induced subspace topology.

		\item Given a hypercovering $(K,\U)$, we refer to the
			simplicial set $K$ as the {\em spine} of the
			hypercovering and to the diagram $\U$ as the {\em
			hypercovering diagram}.

        \item It will often be convenient to complete a hypercovering diagram to a cone
            \[
                \U^{+}: (\Delta/K)^{\lhd} \to \O(X)^{\op}
            \]
            by assigning to the tip of the cone, labelled naturally by the
            empty simplex $\emptyset$, the open subset $X \subset X$. 

        \item Our ad-hoc formulation of the hypercovering condition
            \eqref{eq:hypercovering} admits a more abstract (and general)
            formulation in terms of coskeleta, but the form we have chosen
            seems most reasonable for our purposes. 
	\end{enumerate}
\end{rem}

\begin{exa}
	\label{ex:usual_covering}
	Let $X$ be a topological space, let $I$ be a set, and let $\{U_i\}_{i \in
	I}$ be an open covering of $X$. Define $K$ to be the nerve of the category
	with set of objects $I$ and a unique morphism between any pair of objects
	(so that $K_n = I^{n+1}$). Then the association
	\[
		(i_0, i_1, ..., i_n) \mapsto U_{i_0} \cap U_{i_1} \cap ... \cap U_{i_n}
	\]
	extends to a functor $\Delta/K \to \O(X)^{\op}$ which is a hypercovering
	of $X$. Indeed, the required equality \eqref{eq:hypercovering}, for $n>0$,
	amounts to
	\[
		U_{(i_0, i_1, ..., i_n)} = U_{i_0} \cap U_{i_1} \cap ... \cap U_{i_n}.
	\]
\end{exa}

Note that, in Example \ref{ex:usual_covering}, since $K$ is the nerve of a
contractible groupoid, its geometric realization $|K|$ is contractible. The spine of a
general hypercovering does, of course, not need to be contractible.

\begin{defi}
	A presheaf
	\[	
	    \F: \B^{\op} \to \C
	\]
	is called a $\B$-{\em hypersheaf} if, for every basis open $B \in \B$
	and every hypercovering $(K,\U)$ of $B$ with values in $\B$, the composite diagram
	\begin{equation}
        \label{eq:hyperdescent}
		\F \circ \U^{+}: \Delta^+/K \to \C 
	\end{equation}
	is a limit cone. 
	
	In particular, if $\B=\O(X)$, we call $\F$ a hypersheaf on $X$.
\end{defi}

For any hypercovering $(K,\U)$ of $X$, we will construct a hypercovering
\[
    (K^{\B},\U^{\B}:\Delta/K^{\B} \to \B^{\op})
\]
of $X$ by elements of $\B$, which is a refinement of $(K,\U)$: It comes equipped with a map of simplicial sets $\pi:K^{\B} \to K$, such that, for $\sigma \in K^{\B}$, we have $\U^{\B}(\sigma) \subset \U(\pi(\sigma))$.

An element of $K^{\B}_n$ is given by
\begin{enumerate}
    \item an $n$-simplex $\sigma \in K_n$,
    \item a map $O:\P([n])\backslash \{\emptyset\} \to \B^{\op}$,
\end{enumerate}
such that the following
condition holds:
\begin{itemize}
    \item
    For every nonempty subset $A \subset [n]$, we have $O(A)
    \subset \U(\iota_A^*(\sigma))$, where $\iota_A: \langle m \rangle \to \langle n
    \rangle$ denotes the unique increasing map with $\im \iota_A =A$.
\end{itemize}

The functoriality of $K^{\B}$ is induced by that of $K$ and $\Delta$, and
$\pi:K^{\B} \to K$ is the natural projection map. The functor $\U^{\B}:
\Delta/K^{\B} \to \B^{op}$ is given by
\[
    (\sigma,O) \mapsto O([n])
\]
This yields a hypercovering of $X$ since $(K,\U)$ is
a hypercovering and $\B$ is a basis.

The set $K^{\B}_n$ can be naturally equipped with the structure of a category
(in fact a poset): There is a morphism $(\sigma,O) \to (\tau,U)$ if
$\sigma=\tau$ and $U(A) \subset O(A)$ for all $A \subset [n]$. We denote by
$\tilde{K}^{\B}$ the simplicial category comprising these posets. There is a
again a functor $\tilde{U}^{\B}:\Delta/\tilde{K}^{\B} \to \B^{\op}$ given as
above. There is also a commutative triangle:
\[
\begin{tikzcd}
\Delta/K^{\B} \arrow[r] \arrow[rd,"\Delta/\pi"'] & \Delta/\tilde{K}^{\B} \arrow[d,"\Delta/\pi"] \\
                                   & \Delta/K                       
\end{tikzcd}
\]

Similarly, we also introduce $\B$-refinements $S(K)^{\B}: \Fin^{\op} \to \Set$ for the symmetrized version. An element of $S(K)^{\B}_n$ consists of
\begin{enumerate}
\item $[\sigma,f] \in S(K)_n$,
\item a map $O:\P([n])\backslash\{\emptyset\} \to
\B^{\op}$,
\end{enumerate}
such that:
\begin{itemize}
    \item For all nonempty $A \subset \P([n])$, $O(A) \subset
\U(\iota_{f(A)}^*(\sigma))$, where $\iota_{f(A)}$ is the increasing map
with image $f(A)$.
\end{itemize}

The diagram $S(K)^{\B}$ is again naturally the underlying functor of the diagram of poset categories $S(\tilde{K})^{\B}$, which contains an arrow $([\sigma,f],O) \to ([\tau,g],U)$ if and only if $[\sigma,f]=[\tau,g]$ and $U(A) \subset O(A)$ for all nonempty $A \subset [n]$.

There is further a functor
\[
    \tilde{U}^{\B}: \Fin/\tilde{S(K)^{\B}} \to \B^{\op},([\sigma,f],O) \mapsto O([n])
\]
which induces all the others:
\[\begin{tikzcd}
	{\Delta/K^{\B}} & {\Delta/\tilde{K}^{\B}} \\
	{\Fin/S(K)^{\B}} & {\Fin/S(\tilde{K})^{\B}} & {\B^{\op}}
	\arrow[from=1-1, to=1-2]
	\arrow[from=1-1, to=2-1]
	\arrow[from=1-2, to=2-2]
	\arrow[from=2-1, to=2-2]
	\arrow["{\tilde{\U}^{\B}}", from=2-2, to=2-3]
\end{tikzcd}\]

The following two lemmas are concerned with the arrows in this square.
\begin{lem}\label{lem:symmetrization_coinit}
The functors $\Delta/K^{\B} \to \Fin/S(K)^{\B}$ and $\Delta/\tilde{K}^{\B} \to \Fin/S(\tilde{K})^{\B}$ are coinitial.
\end{lem}
\begin{proof}
There is a pullback square
\[\begin{tikzcd}
	{\Delta/\tilde{K}^{\B}} & {\Fin/S(\tilde{K})^{\B}} \\
	{\Delta/K} & {\Fin/S(K)}
	\arrow[from=1-1, to=1-2]
	\arrow[from=1-1, to=2-1]
	\arrow[from=1-2, to=2-2]
	\arrow[from=2-1, to=2-2]
\end{tikzcd}\]
where the bottom arrow is coinitial by Lemma \ref{prop:sym_unit_coinit}. As pullback along Cartesian fibrations preserves coinitiality, the claim follows. The argument for the discrete version of the construction is analogous.

\end{proof}
\begin{lem}\label{lem:sym_categorification_coinit}
The functor $i:\Fin/S(K)^{\B} \to \Fin/S(\tilde{K})^{\B}$ is coinitial.
\end{lem}
\begin{proof}
We will show that, for $([\sigma,f],O) \in S(\tilde{K})^{\B}_n$ the slice category $(\Fin/S(K)^{\B})/([\sigma,f],O)$ is weakly contractible.
Let $H:\Fin^{\op} \to \Set$ denote the functor with $H_m$ consisting of arrows
\[\begin{tikzcd}
	{([\tau,g],U)} & {([\sigma,f],O)}
	\arrow["h", from=1-1, to=1-2]
\end{tikzcd}\]
in $S(\tilde{K})^{\B}$, with $h:\langle m \rangle \to \langle n \rangle$. The functoriality is given by precomposition. The slice category under consideration is then isomorphic to $\Fin/H$. By Corollary \ref{cor:forget_fin}, we have
\[
    |\Fin/H| \simeq |\Delta/H|.
\]
To show that $|\Delta/H|$ is contractible, we will show that the underlying simplicial object of $H$ is a trivial Kan complex.

For $m \geq 1$, a diagram $\partial \Delta^m \to H$ induces in particular, via composition, a diagram
\[\partial \Delta^m \to  H \to \Fin(-,\langle n \rangle),
\]
which admits a unique filling $g: \langle m \rangle \to \langle n \rangle$. The $m$-boundary $\partial \Delta^m \to H$ further contains the data of a map
\[
    U:\P([m])\backslash\{\emptyset,[m]\} \to \B^{\op}.
\]
We extend this functor by setting $U([n])=O(\im g)$. The $m$-simplex
\[\begin{tikzcd}
	{(g^*[\sigma,f],U)} & {[\sigma,f]}
	\arrow["g", from=1-1, to=1-2]
\end{tikzcd}\]
is then a well-defined filling of the given boundary.
\end{proof}

\begin{thm} Let $X$ be a topological space and let $\B$ be a basis of the
    topology. For any $\infty$-category $\C$ with limits, the restriction
    \[
        \Hyp(X,\C) \to \Hyp(\B,\C)
    \]
    is an equivalence of $\infty$-categories.
\end{thm}
\begin{proof}
Let $\F: \O(X)^{\op} \to \C$ be a functor whose restriction $\F|\B^{\op}$ is a
$\B$-hypersheaf. It will suffice to show that $\F$ is a hypersheaf if and only
if it is a right Kan extension of $\F|\B^{\op}$.

First, assume that $\F$ is a hypersheaf. Let $U \subset X$ be open. We denote
by $(K,\U)$ the trivial hypercovering of $U$, i.e. $K=\Delta^0$ and
$\U(*)=U$. By the hypersheaf condition, restriction naturally identifies
$\F(U)$ with the limit of the diagram
\[
\begin{tikzcd}
\Delta/K^{\B} \arrow[r, "\U^{\B}"] &U/\B^{\op} \arrow[r, "\F"] & \C.
\end{tikzcd}
\]
We conclude the pointwise formula for a right Kan extension by showing that
\[
    \U^{\B}:\Delta/K^{\B} \to U/\B^{\op}
\]
is coinitial: By Quillen's Theorem A, we
must show that for an open set $B \in \B$ with $B \subset U$, the full
subcategory of $\Delta/K^{\B}$ spanned by the simplices $(\sigma,O)$ with $\U^{\B}(\sigma,O)
\supset B$ is weakly contractible. This category is naturally isomorphic to the
category of simplices of the simplicial subset
\[
    K_{B \subset}^{\B} \subset K^{\B}
\]
containing precisely those simplices. As any map
$\partial \Delta^n \to K_{B \subset}^{\B}$ admits a filling to an $n$-simplex
via the assignment $id_{\Delta^n} \mapsto B$, we find that $K_{B \subset}^{\B}$
is a trivial Kan complex, which proves the claim.

Conversely, suppose that $\F$ is a right Kan extension of the $\B$-hypersheaf
$\F|\B^{\op}$. For any hypercover $(K,\U)$ of an open subset $U \subset X$, we
must show that $\F \circ \U^{+}:\Delta^+/K \to \C$ is a limit diagram. To this
end, we first show that the diagram
\[
\begin{tikzcd}
\Delta/\tilde{K}^{\B} \arrow[rd, "\tilde{\U}^{\B}"] \arrow[d, "\Delta/\pi"'] &                                       &    \\
\Delta/K \arrow[r, "\U"']                                            & \O(X)^{\op} \arrow[r] \arrow[r, "\F"] & \C
\end{tikzcd}
\]
exhibits $\F \circ \U$ as the right Kan extension of $\F \circ \tilde{\U}^{\B}$
along $\Delta/\pi$. As $\Delta/\pi$ is a Cartesian fibration, this amounts to
showing that, for $\sigma \in K_n$, we have 
\begin{equation}\label{eq:cat_hy_computes_limit}
\F(\U(\sigma)) \simeq \lim_{(\Delta/\pi)^{-1}(\sigma)} \F \circ \U^{\B}.
\end{equation}
The restriction $\U^{\B}:(\Delta/\pi)^{-1}(\sigma) \to (\B/\U(\sigma))^{\op}$
is coinitial, since for $B \in (\B/\U(\sigma))^{\op}$, the slice category
$(\Delta/\pi)^{-1}(\sigma)/B$ has a final object given by
$(\sigma,O:\P([n])\backslash \{\emptyset\} \to \B^{\op}, A \mapsto B)$.

Thus, \eqref{eq:cat_hy_computes_limit} reduces to the equivalence
\begin{equation}
    \F(\U(\sigma)) \simeq \lim_{B \in (\B/\U(\sigma))^{\op}}\F(B),
\end{equation}
which holds since $\F$ is a right Kan extension of its restriction to $\B^{\op}$.

By the above, it now suffices to show that $\F \circ
(U^{\B})^+:\Delta^+/\tilde{K}^{\B} \to \C$ is a limit diagram. As $\F$ is a
right Kan extension of its restriction to $\B^{\op}$, we have
\begin{equation}
    \F(U) \simeq \lim_{B \in (\B/U)^{\op}}\F(B),
\end{equation}
so that it in fact suffices to show that the diagram
\[
    \begin{tikzcd}
        \Delta/\tilde{K}^{\B} \arrow[d, "\U^{\B}"'] \arrow[rd, "\F \circ \U^{\B}"] &    \\
        (\B/U)^{\op} \arrow[r, "\F"]                                               & \C
    \end{tikzcd}
\]
exhibits $\F$ as the right Kan extension of $\F \circ \U^{\B}$ along $\U^{\B}$.
For $B \in (\B/U)^{\op}$, let
\[
    B/(\Delta/\tilde{K}^{\B}) \subset
\Delta/\tilde{K}^{\B}
\]
denote the full subcategory spanned by simplices $(\sigma,O)$
such that $\U^{\B}(\sigma,O) \subset B$. The above then amounts to showing that
\begin{equation}\label{eq:refined_cover_limit}
    \F(B) \simeq \lim_{B/(\Delta/\tilde{K}^{\B})} \F \circ \U^{\B}.
\end{equation}
Let
\[
    K^{\B}_{\subset B} \subset K^{\B}
\]
denote the full simplicial subset containing tuples $(\sigma,O)$ such that $O(A) \subset B$ for all subsets $A$. Similarly, let
\[
    \tilde{K}^{\B}_{\subset B} \subset \tilde{K}^{\B}
\]
denote the full simplicial subcategory containing those simplices. We have a commutative diagram
\[
\begin{tikzcd}
    \Delta/K^{\B}_{\subset B} \arrow[d, "i"'] \arrow[rd]        &                              &    \\
    \Delta/\tilde{K}^{\B}_{\subset B} \arrow[d, "j"'] \arrow[r] & (\B/B)^{\op} \arrow[r, "\F"] & \C \\
    B/(\Delta/\tilde{K}^{\B}) \arrow[ru, "\tilde{\U}^{\B}"']                       &                              &   
\end{tikzcd}
\]
Note that $(K^{\B}_{\subset B},\U^{\B})$ is a $\B$-hypercovering of $B$. As
$\F$ is a $\B$-hypersheaf, this implies that
\[
    \F \circ
(\U^{\B})^+:\Delta^+/K^{\B}_{\subset B} \to \C
\]
is a limit diagram. We conclude
the proof of \eqref{eq:refined_cover_limit} by showing that this limit
propagates along the inclusions in the above diagram.

We first show that $\F \circ \tilde{\U}^{\B}$ is a right Kan extension of its
restriction along $j$, implying that the two functors have the same limit. For
$(\sigma,O) \in B/(\Delta/\tilde{K}^{\B})$ the inclusion of the full
subcategory $A \subset(\sigma,O)/(\Delta/\tilde{K}^{\B}_{\subset B})$ spanned
by the morphisms 
\begin{equation}
\begin{tikzcd}
(\sigma,O) \arrow[r, "id"] & (\sigma,O')
\end{tikzcd}
\end{equation}
is coinitial since the relevant overcategories have final objects. The
restriction
\[
    \tilde{\U}^{\B}: A \to (\B/\tilde{\U}^{\B}(\sigma,O))^{\op}
\]
is coinitial, since
for every $V \in (\B/\\tilde{U}^{\B}(\sigma,O))^{\op}$ the slice category has a final
object given by the constant map with value $V$. Therefore,
\[ \F(\tilde{\U}^{\B}(\sigma,O)) \simeq \lim_{ (\sigma,O)/(\Delta/\tilde{K}^{\B}_{\subset B}) } \F
\circ \tilde{\U}^{\B}
\]
and the restriction along $j$ is a Kan extension.

Finally, we show that $\F \circ (\U^{\B})^+|\Delta^+/\tilde{K}^{\B}_{\subset
B}$ is a limit diagram. Slight variations of Lemma
\ref{lem:symmetrization_coinit} and Lemma \ref{lem:sym_categorification_coinit}
yield a commutative diagram 
\[\begin{tikzcd}
	{\Delta/K^{\B}_{\subset B}} & {\Delta/\tilde{K}^{\B}_{\subset B}} \\
	{\Fin/S(K)^{\B}_{\subset B}} & {\Fin/S(\tilde{K})^{\B}_{\subset B}} & {\B^{\op}} & \C
	\arrow["i", from=1-1, to=1-2]
	\arrow[from=1-1, to=2-1]
	\arrow[from=1-2, to=2-2]
	\arrow[from=2-1, to=2-2]
	\arrow["{\tilde{\U}^{\B}}", from=2-2, to=2-3]
	\arrow["\F", from=2-3, to=2-4]
\end{tikzcd}\]
where the three morphisms other than $i$ in the square are coinitial. We conclude that
\[
    \F(B) \simeq \lim_{\Delta/K^{\B}_{\subset B}} \F \circ \U^{\B} \simeq \lim_{\Delta/\tilde{K}^{\B}_{\subset B}} \F \circ \tilde{\U}^{\B}.
    \]
\end{proof}

\newpage
\bibliographystyle{alpha} 
\bibliography{refs}

\begin{thebibliography}{{Lur}17}

\bibitem[Aok23]{aoki:ttg}
K.~Aoki.
\newblock Tensor triangular geometry of filtered objects and sheaves.
\newblock {\em Math. Z.}, 303:Paper No. 62, 27pp, 2023.

\bibitem[BGH20]{bgh:exodromy}
C.~Barwick, S.~Glasman, and P.~Haine.
\newblock Exodromy.
\newblock {\em arXiv preprint arXiv:1807.03281}, 2020.

\bibitem[Lur09]{lurie:htt}
J.~Lurie.
\newblock {\em Higher topos theory}, volume 170 of {\em Annals of Mathematics Studies}.
\newblock Princeton University Press, Princeton, NJ, 2009.

\bibitem[{Lur}17]{lurie:ha}
J.~{Lurie}.
\newblock {Higher Algebra}.
\newblock {\em preprint}, September 2017.

\bibitem[Lur18]{lurie:sag}
J.~Lurie.
\newblock Spectral algebraic geometry.
\newblock {\em preprint}, February 2018.

\bibitem[Ver72]{verdier:SGA4V}
J.~L. Verdier.
\newblock Cohomologie dans les topos.
\newblock In {\em Th{\'e}orie des Topos et Cohomologie Etale des Sch{\'e}mas}, pages 1--82, Berlin, Heidelberg, 1972. Springer Berlin Heidelberg.

\end{thebibliography}

\end{document}